\documentclass[11pt]{article} 

\usepackage[utf8]{inputenc} 

\usepackage{geometry} 
\usepackage{amsmath}
\usepackage{graphicx} 
\usepackage{color}

\usepackage{booktabs} 
\usepackage{array} 
\usepackage{paralist} 
\usepackage{verbatim} 
\usepackage{subfig} 

\usepackage{fancyhdr} 
\usepackage{sectsty}
\allsectionsfont{\sffamily\mdseries\upshape} 

\usepackage{amsfonts}
\newcommand{\QED}{\hfill\rule{0.5em}{0.5em}\\}
\newcommand{\QEDc}{\hfill$\Box$\\}
\newcommand{\Mod}[1]{\ (\mathrm{mod}\ #1)}
\newcommand{\ILP}{ILP$^{\mbox{\tiny D}}_{\mbox{\tiny Real}}$ }

\usepackage{amsmath}

\newtheorem{theorem}{Theorem $\!\!$}[section]
\newtheorem{lemma}[theorem]{Lemma $\!\!$}
\newtheorem{remark}[theorem]{Remark $\!\!$}

\newenvironment{proof}[1][Proof]%
{\begin{description}\item[\noindent\textbf{#1}:]}
{\QED\end{description}}

{\begin{description}\item[\noindent\textbf{#1}:]}
{\QEDc\end{description}}

\newenvironment{keyword}{\textbf{Keywords:}}{}

\title{A Practical Algorithm for the Computation of the Genus}
\author{Gunnar Brinkmann}

\begin{document}
\maketitle

\begin{abstract}
  We describe a practical algorithm to compute the (oriented) genus of a graph, give results of the program implementing this
  algorithm, and compare the performance to existing algorithms.
  The aim of this algorithm is to be fast enough for many applications instead of focusing on the theoretical asymptotic complexity.\\
  Apart from the specific problem and the results, the article can also be seen as an example how some design principles used to
  carefully develop and implement standard backtracking algorithms can still result in very competitive programs.

\end{abstract}


\section{Introduction}

Algorithms are studied in two different ways. The first one is as underlying methods of computer programs. These algorithms only
come to their right when implemented as a program and used as a tool. The second way is as objects of research themselves.
In this case the emphasis lies on determining the asymptotic complexity of a problem -- that is: of optimal
algorithms solving the problem -- often without
the intention or realistic possibility to transform the described algorithms to a computer program that can be used.
Even if such algorithms could be implemented and used, many would be extremely inefficient for real problem sizes
and outperform standard algorithms only for problem sizes far beyond the limit where either of them can be used. Of course
there are also some nice cases -- e.g. the linear time planarity algorithm described in \cite{wendy_linear_planar} -- where both
aspects meet and algorithms with the best asymptotic behaviour also perform well in practice.

The general difference between these two approaches can very well be illustrated at the example of the problem of determining the genus of a graph,
which is defined as the genus of the smallest orientable 2-manifold so that the graph can be embedded on its surface
without crossing edges. The problem is NP-complete \cite{genus_NPcomplete}, but for any fixed $g$ there is a linear time
algorithm that can compute the genus $g'\le g$ or decide that it has genus larger than $g$ \cite{linear_genus}. On the other hand, there is
no algorithm implemented and available that is at least guaranteed to compute the genus of a single sparse graph with, say,
80 vertices, or determine that it has genus larger than 20, in one year of CPU time.

Even lately papers have been published that theoretically determine the genus of specific relatively small
graphs or describe algorithms fine tuned for these graphs
\cite{genus_circulants}\cite{conder_genus}\cite{genus_gray}, but in addition to that, researchers have also started to develop general purpose genus
computation programs. In \cite{genusILP1} such an algorithm based on an integer linear programming approach was
published. Later, an improved approach described in \cite{genusILP2} was -- although also being a general purpose algorithm -- able to achieve
many of the results formerly obtained by individual research, automatically in a reasonable amount of time. Unfortunately these
programs are neither publicly available nor easy to use.

In this article we will describe an algorithm that -- in spite of also being exponential already for small genus --
clearly outperforms these approaches and is widely usable. The program based on this algorithm has the
options to compute the genus of a graph, one or all minimum genus embeddings, one or all embeddings on
an orientable surface of given genus, or to filter large lists of graphs for those with genus at most or at least a given bound.
When choosing for {\em all embeddings}, for graphs with a nontrivial symmetry group
isomorphic embeddings can be generated, but no two embeddings that are (labelled) mirror
images are generated.
We use a carefully designed backtracking algorithm.

\section{The algorithm}

We assume all input graphs to be simple connected graphs. The embeddings are constructed by interpreting each
undirected edge as a pair of oppositely directed edges. We build a {\em rotation system},
that is a cyclic ordering of all directed edges starting in a vertex and interpret this ordering as clockwise. Faces of an embedded
graph are defined by the usual face tracing algorithm starting from a directed edge $(v,w)$ and constructing
the face containing $(v,w)$ by going to the inverse edge
$(w,v)$, and then to the next edge in the orientation around $w$. This process is repeated until being back at $(v,w)$.
A face $f$ is thus a set of oriented edges. The set of all vertices contained in one of the directed edges of a face $f$
will be denoted by $f^v$, $f(e)$ will denote the number of directed edges in the face containing the directed edge $e$,
$F$ will denote the set of all such faces, $V$ the set of vertices, and $E$ the set of edges.
The genus $g$ of the embedded graph is $g=1+(|E|-|V|-|F|)/2$. 

\bigskip

\subsection*{Preprocessing:}

Vertices of degree $1$ are irrelevant when computing the genus -- they
can simply be removed without any impact on the genus. Similarly
vertices of degree $2$ can be replaced by an edge connecting their two
neighbours. If this operation produces a double edge, the new edge can be
removed too without changing the genus of the graph. Except when all
embeddings of a graph on a certain genus must be computed and there are at least
three vertices, these operations are recursively applied before the
real computation of the genus begins. This means that e.g. cycles, trees or
complete bipartite graphs $K_{2,n}$ are all reduced to a single vertex.
After having computed an embedding, reduced vertices 
are restored. For graphs with minimum degree at least $3$ -- which is almost
always the case when mathematical research about the genus is done -- this
preprocessing step has of course no impact.

When computing the genus, the algorithm works by first searching for plane embeddings, then embeddings with genus 1, etc.\ until
an embedding is found.
The upper bound for the genus of the embedding that is to be constructed is used in the recursive routine
embedding edges. When trying to embed the graph in genus $g>0$, it has already been determined
that there are no embeddings of genus at most $g-1$ and the first embedding of genus $g$ determines the genus of the graph.
Sometimes -- this also depends on chance -- such an embedding can be found relatively fast and the real bottleneck is the complete search
for embeddings of genus $g-1$. By computing a lower bound on the genus, sometimes expensive complete searches can be avoided, but
the lower bound must be fast to compute in order to have an advantage over the complete search. We will now first describe
a method to compute a (cheap) lower bound:

\bigskip

\subsection*{Computing a lower bound for the genus:}

When embedding a graph $G=(V,E)$, the values of $|V|$ and $|E|$ are fixed, so a minimum genus embedding is in fact an embedding with a maximum number of faces
and if $f'$ is an upper bound on the number of faces in any embedding then $g'=\lceil 1+\frac{|E|-|V|-f'}{2}\rceil $ is a lower bound on the genus.

A trivial upper bound on the number of faces is $\frac{2|E|}{3}$ as all faces have at least three edges. This lower bound is practically
free and is always computed and used. Instead of the constant value $3$, except for trees
one could also use the girth of the graph, but tat would also have to be computed.
The following methods give a better bound if there are few cycles of minimum length.

For a given embedding, let $s[]$ denote the vector of size $2|E|$  indexed from 1 to $2|E|$ containing all values $f(e)$ of directed edges $e$
in non-decreasing order.
Then $|F|=F(s)=\sum_{i=1}^{2|E|}(1/s[i])$. A vector $s'[]$ of size $2|E|$ with $s'[i]\le s[i]$ for $1\le i\le 2|E|$,
is said to be dominated by $s[]$. For a vector $s'[]$ dominated by $s[]$ we have $F(s')=\sum_{i=1}^{2|E|}(1/s'[i])\ge f$.

We call a cyclic sequence $e_0,\dots ,e_{k-1}$ of $k$ pairwise distinct directed edges a {\em facial-like walk} if
and only if for $0\le i<k$ the starting vertex of $e_{i+1\Mod{k}}$ is the end vertex of $e_i$ and $e_{i+1\Mod{k}}$ is the inverse $(e_i)^{-1}$
of $e_i$ if and only if the degree of the end vertex of $e_i$ is one.
A first approximation $s_0[]$ of $s[]$
is obtained by taking for each directed edge $e$ the length $f_w(e)$ of the shortest facial-like walk containing $e$. The value of $f_w(e)$ can be easily computed by
a Breadth First Search.

As each facial walk in an embedded graph
is also a facial like walk, we see immediately that the non-decreasing sequence $s_0[]$ 
is dominated by $s[]$, as $f_w(e)\le f(e)$ for each directed edge $e$. Especially $s_0[2|E|]\le s[2|E|]$ and as in $s[]$ at least
$s[2|E|]$ edges -- all directed edges in a longest facial walk -- have value $s[2|E|]$, we can replace the last $s_0[2|E|]$ values
of $s_0[]$ with $s_0[2|E|]$ and get another sequence $s_1[]$ dominated by $s[]$. We use $F(s_1)$ as a first nontrivial upper bound on the
number of faces.

In fact the length of the shortest facial-like walk is the same for a directed edge and its reverse, but unless the graph is a cycle,
one facial-like walk that does not also contain the reverse edge, can only form a face for at most one of them. This
observation might lead to a better approximation, but
in order to keep the computation of the approximation easy and fast, the length of the shortest facial-like walk is used for a directed edge and its inverse.

\medskip

An angle $\alpha$ of a face is a pair of directed edges, following each
other in the facial walk. The central vertex of the angle is the
endpoint of the first edge -- so except when this vertex has degree
$1$ it is the only common vertex of the two edges. In what follows we
use that for an edge $e$ and its inverse $e^{-1}$ we have
$f_w(e)=f_w(e^{-1})$.\\

Instead of summing over all edges, we can
sum over all angles. With $f(\alpha)$ the size of the face that
contains $\alpha$ and $A(v)$ the set of all angles with central vertex
$v$, we have $|F|=\sum_{v\in V}(\sum_{\alpha \in A(v)} 1/f(\alpha))$. If
for a vertex $v$ the sequence $s'_v[1],\dots ,s'_v[\deg(v)]$ is the
non-decreasing sequence of all $f(\alpha)$ with $\alpha \in A(v)$,
then $|F|=\sum_{v\in V}(\sum_{1\le i \le \deg(v)} 1/s'_v[i])$. Taking
for each vertex $v$ a vector dominated by $s'_v[]$ we again get an
upper bound on $|F|$. If we take for a vertex $v$ and each angle
$\alpha\in A(v)$ instead of $f(\alpha)$ the value $\max\{f_w(e),f_w(e')\}$, with $e,e'$
the edges in the angle, we get a non-decreasing sequence $s'_{0,v}[]$ 
dominated by $s'_v[]$. If $s'_{1,v}[]$ is the non-decreasing sequence of values
of $f_w(e)$ with $e$ starting at $v$, then we define 
$s'_{2,v}[]=s'_{1,v}[2], s'_{1,v}[3], \dots , s'_{1,v}[\deg(v)], s'_{1,v}[\deg(v)]$. So we remove the smallest value
of $s'_{1,v}[]$ and add a copy of the largest value.

\begin{remark}\label{rem:dominatev}

  Let $G=(V,E)$ be an embedded graph. Then for each vertex $v\in V$ the sequence $s'_v[]$
  dominates $s'_{2,v}[]$.

\end{remark}

\begin{proof}

  We know that $s'_v[]$ dominates $s'_{0,v}[]$. We will show that $s'_{0,v}[]$ dominates $s'_{2,v}[]$.
   As the maximum
   values of $s'_{0,v}[]$, $s'_{1,v}[]$ and $s'_{2,v}[]$ are the same, it is sufficient to prove
   $s'_{0,v}[i] \ge s'_{2,v}[i]=s'_{1,v}[i+1]$ for $i<\deg(v)$.
   Let $\alpha_1, \dots ,\alpha_i$ be the angles (that is: pairs of edges) determining the values
   $s'_{0,v}[1], \dots ,s'_{0,v}[i]$ and $S_{i,v}$ be the set of all directed edges starting at $v$, so that $e$ or $e^{-1}$ is in at least one of
   these angles. Then the value of $s'_{0,v}[i]$ is $\max \{ f_w(e)| e\in S_{i,v}\}$ (here we use that $f_w(e)=f_w(e^{-1})$)
  and as $|S_{i,v}|\ge i+1$, we have that $s'_{0,v}[i]\ge s'_{1,v}[i+1]$.

  \end{proof}

We use $\sum_{v\in V}(\sum_{1\le i \le \deg(v)} 1/s'_{2,v}[i])$ as a second nontrivial upper bound on the number of faces.

\bigskip

These upper bounds on the number of faces and the corresponding lower bounds on the genus are relatively fast to compute.
Nevertheless they do not always speed up the program. Especially for small graphs or small genus
they can even slow down the program, as the embedding algorithm can exclude low genus
embeddings very fast. While for few small graphs this is no problem, for large lists of small graphs it can be a problem.
Of course the bounds can never slow down the program much, but can speed it up a lot:

\medskip

In this article, all running times for the C-program {\em multi\_genus} implementing the algorithm described here are on an
{\em Intel Core i7-9700 CPU @ 3.00GHz} (running on one core at 4.4-4.7 Ghz).
Examples for the impact of the computation of a lower bound when computing the genera of graphs are:\\

All bipartite graphs on 14 vertices with degrees between 5 and 6 (73 graphs, genus 3 to 5): without lower bound 60.9 seconds, with lower bound 0.035 seconds. \\
All cubic graphs on 22 vertices (7,319,447 graphs, genus 0 to 3):  without lower bound 300 seconds, with lower bound 364 seconds.\\
Checking 1,000,000 random cubic graphs on 50 vertices, generated by {\em genrang} (which is part of the nauty-package \cite{gtools}) for being planar: without lower bound 18.5 seconds, with lower bound 56.2 seconds.\\
Checking the same 1,000,000 random cubic graphs on 50 vertices for having genus at most $1$: without lower bound 144.8 seconds, with lower bound 85.5 seconds.\\

\bigskip

The default is that the nontrivial bounds are used, but the use can be switched off by an option to multi\_genus.

\subsection*{Constructing an embedding}

We begin by relabeling the graph in a BFS way. The time necessary to compute the genus can differ  a lot for isomorphic graphs
depending on the labelling. In some cases a BFS labelling results in a better performance, in others it slows down the program. We have chosen for
the BFS labeling as the results for different, but isomorphic, input graphs often differ less when always using such a labeling.
An example showing the large differences that can still occur can be seen when
computing a genus 7 embedding (that is a minimum genus embedding) of
$C_3\square C_3 \square C_3$. Taking the first graph of the file \verb+ucay27_05_k=06+ provided by Gordon Royle in his
list of Cayley graphs (and doing BFS), it takes $0.19$ seconds to find an embedding, taking the same graph from a program constructing
cartesian products and not doing BFS, it takes $6.2$ seconds. Taking the graph from the second source and doing BFS, it takes
$281$ seconds. So even when relabeling the graph in a BFS manner the time still depends on the labeling of the input graph.

The algorithm works by first greedily embedding a subgraph so that for each embedding or its mirror image, the induced
embedding of this subgraph is the one constructed. It has genus $0$.
Then we add one edge at a time to the embedding.
If the maximum degree is smaller than $3$, the graph is a path or a cycle,
both of which can be uniquely embedded and only in the plane.
Otherwise we construct the initial subgraph by taking a vertex with minimum degree among all vertices with degree at least $3$,
embedding this vertex and three of its edges in an arbitrary way (thereby fixing the orientation) and greedily extending the three edges
-- one after the other -- to paths until they cannot be made longer. The result of this construction forms the root of the recursion tree.

In branch and bound algorithms, the performance is often improved
if one manages to reduce the branching at every node of the recursion tree.
This is not a mathematical
theorem, but more a rule of thumb, as in some cases more branching might be beneficial if it allows earlier bounding. Nevertheless
in our case we have chosen to take small branching as the base (but not only) criterion for the order in which the edges are inserted.
As an expensive choice of the next edge to insert is sometimes more costly than more branching, we work in three parts:

\begin{description}
\item[(i)] Before the recursion starts, the edges that are still to be embedded
  are sorted as $\{x_1,y_1\},\dots ,\{x_k,y_k\}$, so that with $S_i$ the initial subgraph
  with edges $\{x_1,y_1\},\dots ,\{x_i,y_i\}$ added, 
  for $1\le i <k$, we have for $ \{x_{i},y_{i}\}$ that at least one of $ x_{i},y_{i}$ is in $S_i$ and
  $\deg_{S_{i-1}}(x_{i})\cdot \deg_{S_{i-1}}(y_{i})\le \deg_{S_{i-1}}(x_{j})\cdot \deg_{S_{i-1}}(y_{j})$
    for all $i<j\le k$ for which at least one of $ x_{j},y_{j}$ is in $S_i$.
  If one of $\deg_{S_{i-1}}(x_{i}),\deg_{S_{i-1}}(y_{i})$ is $0$, then also
  $\deg_{S_{i-1}}(x_{i})+\deg_{S_{i-1}}(y_{i})\le\deg_{S_{i-1}}(x_{j})+\deg_{S_{i-1}}(y_{j})$ for all $j$ for which one
  of $deg_{S_{i-1}}(x_{j}),\deg_{S_{i-1}}(y_{j})$ is $0$ and the other is in $S_i$.
  Informally speaking: first all edges leading from the already embedded part to not yet included vertices are added and
  then the edges with both endpoints in the embedded subgraph are added. In each case we choose an edge that has the smallest
  number of possibilities how it can be added if there are no other restrictions.
\end{description}

During the recursion we always have an upper bound \verb+max_genus+ on the genus. As long as this bound is not reached,
edges can be inserted in all possible angles -- sometimes increasing the genus and sometimes not. In order to reach this bound as
fast as possible and therefore be able to reduce the branching, we check at each node of the recursion tree
whether there is an edge that cannot be embedded
in any existing face of the partial embedding -- we call such an edge a {\em critical edge} --
and therefore always increases the genus. If there is such an edge and the partial embedding has already genus
\verb+max_genus+, we can backtrack, otherwise the first such edge in the list gets priority above other edges
that do not have this property and is inserted first.

We distinguish two cases:

\begin{description}
\item[(ii)] If the recursion is still close to the root node and the genus of the partial embedding is still smaller than \verb+max_genus+, we have
  relatively few nodes and the impact of a smaller branching is
  large. In this situation, also more expensive tests can pay and we do not only check for the existence of a critical edge,
  but do in fact look for an edge which has the smallest number of faces into which it can be embedded and take such an edge as the
  next one to be embedded. Among all edges with the same number of faces where they can be embedded, the first one in the
  sorted list is taken. Note that it is possible that an edge can be embedded into a face in more than one way, but this is
  not taken into account when counting the number of faces.

\item[(iii)] Close to the leaves of the recursion tree we have
  many nodes and the impact of a smaller branching is
  small. In this situation, or when the genus of the partial embedding is already \verb+max_genus+,
  we only check for the existence of a critical edge.
\end{description}

The decision when we consider a node of the recursion to be close to the root or close to the leaves has an impact
on the performance, but a simple rule for the optimal moment to switch
is hard to determine. Tests on different kinds of graph
showed that considering nodes where at most half of the edges (edges of the initial tree not counted) are embedded
as close to the root and considering the others as close to a leave is often a good compromise.

The method to find critical edges fast is crucial for the performance. An edge can be embedded into a face if both endpoints
of the edge are in the same face. In the implementation we use bitvectors -- that is: integers of type {\em unsigned long int} (64 bit) or
{\em unsigned \_\_int128} to represent sets of vertices. Especially for graphs with up to $64$ vertices this allows
 to determine whether an edge can be embedded in a face in few CPU cycles -- provided the fact that the set of
end vertices of the edge and the set of vertices of the face are represented as bitvectors. Unless otherwise mentioned,
up to 64 vertices the version using {\em unsigned long int} is used for the timings and the version using {\em unsigned \_\_int128}
for larger graphs on up to 128 vertices.

\begin{lemma}
  Assume that an algorithm to embed a graph $G=(V,E)$ starts with embedding a spanning tree and then 
  inserts the remaining edges
  of a graph step by step, but in each step inserts non-critical edges only if there are no critical edges.
  Let $G'=(V',E')$ be a subgraph with at least one cycle that was embedded by this algorithm and let $e$ be the last
  edge that was inserted in one face $f$ and split it into two faces with vertex sets $f^v_1,f^v_2$.\\
If there is a critical edge $e_c$ for $G'$, then $|e_c\cap (f^v_1\setminus f^v_2)|=|e_c\cap (f^v_2\setminus f^v_1)|=1|$.
\end{lemma}

\begin{proof}
  
  Note first that $G'$ need not have faces with vertex sets $f^v_1,f^v_2$. The faces $f_1,f_2$ can not have been subdivided, as
  they are the result of the last subdivision, but after that subdivision they might have been
  united with other faces (or with each other) when
  an edge with endpoint in two different faces was inserted.

  If there is no critical edge, the statement is trivially true, so assume that  there is a critical edge $e_c$. Let $G_0$ be the
  embedded subgraph into which $e$ was embedded to form $G_1$. As $e$ was inserted into a face, there was no critical edge for $G_0$,
  so $e_c$ could be embedded into a face $f_0$ of $G_0$ (so $e_c\subset f^v_0$).
  After $e$ only critical edges were inserted, so the vertex sets
  of all faces of $G'$ are unions of those of $G_1$. If $f\not= f_0$, $f^v_0\subseteq f'^v_0$ for some
  face $f'_0$ of $G'$, so $e_c$ could be embedded into $f'_0$ and would not be critical. So we have $f= f_0$. As 
  $(f^v_1\setminus f^v_2)\cap (f^v_2\setminus f^v_1)=\emptyset$ it is sufficient to show that $e_c\cap (f^v_1\setminus f^v_2)$ and
  $e_c\cap (f^v_2\setminus f^v_1)$ are both not empty. Assume that (w.l.o.g.) $e_c\cap (f^v_1\setminus f^v_2)=\emptyset$.
  As $f^v_1\cup f^v_2=f^v$, we have $e_c\subset f^v_2$, but as $f^v_2$ is a subset of the vertex set of a face of $G'$,
  $e_c$ would not be critical. So $e_c\cap (f^v_1\setminus f^v_2)\not= \emptyset$.
 
\end{proof}

Finding critical edges is a nontrivial task. The straightforward way is a loop over all edges that still need to be
inserted and inside this loop a loop over all faces of the embedded graph. The previous lemma gives a very cheap criterion
to decide for many edges that they are not critical -- without using the inner loop. In fact one could even make a list of all
candidates for critical edges whenever a face is subdivided, but in the implementation this is not done.

\subsection*{Performance}

There are some programs available, where the exact algorithm is not
published -- e.g. \verb+simple_connected_genus_backtracker+ in the
computer algebra package sage. As a backtracking program it seems to
be related to the algorithm described here. The manual says that it is
{\em an extremely slow but relatively optimized algorithm. This is
  “only” exponential for graphs of bounded degree, and feels pretty
  snappy for 3-regular graphs.}  It also says that {\em $K_7$ may take
  a few days}, while multi\_genus takes less than
$0.001$ seconds for $K_7$.  So we tested it only for cubic graphs, but
already for relatively small vertex numbers, it is very slow,
e.g. more than $24$ minutes for the unique cubic graph with girth 8 on
34 vertices (instead of less than $0.001$ seconds of multi\_genus), so tests on a larger scale were not possible.

The program used in the graph database HoG \cite{HoG} at the moment
(it will be replaced by multi\_genus) is much better. It
is a Java program called {\em MinGenusEmbedder} written by Jasper Souffriau
as a student project and it is also a backtracking
algorithm using branch and bound. That program was also used for
independent tests. For the generation of random graphs we use the program {\em genrang}
\cite{gtools} which allows to restrict the generation to regular graphs of a given degree or to
graphs with a given number of edges. As genrang also generates graphs that are not connected, we filtered them
for connected graphs. If we say that we tested $n$ random graphs generated by genrang, this means that
we generated random graphs by genrang and took the first $n$ connected ones.
In order to have the results completely reproducible, we always fixed
the seed used by genrang to 0.

For 2000 random cubic graphs on 30 vertices, MinGenusEmbedder needed 16.3
seconds (compared to 0.6 seconds of multi\_genus -- a
factor of 27) and for 2000 random cubic graphs on 40 vertices, MinGenusEmbedder
needed 435 seconds (compared to 12.5 seconds, a factor of 34). For
larger degrees the ratio grows. For 30 quartic graphs on 30 vertices MinGenusEmbedder already needs
792 seconds (compared to 6.6 seconds, a factor of 120) and for 30 5-regular graphs on 22 vertices
2,226 seconds (compared to 13.7 seconds, a factor of 163). For 6-regular graphs, testing 30 graphs on
19 vertices already took quite some time: 10.56 hours for MinGenusEmbedder and 2 minutes for
multi\_genus (a factor of 315).
Of course such small samples are not sufficient for
reliable results and we should see these numbers just as a hint what the relation of the running times might be. Unfortunately
the running times do not allow tests on large sets of data.

The fastest published general purpose program is the integer linear programming approach described in
\cite{genusILP2} and implemented in the program {\em \ILP }.
The program \ILP is not publicly available, so we compare the running times for a data set they
used in \cite{genusILP2}: the Rome graphs, which can be downloaded from \verb+http://graphdrawing.org/data.html+.
This set of graphs contains 11,534 graphs with (at least indicated by the file names)
up to 100 vertices which are used e.g. for graph drawing
and are said to come from practical applications. Among these graphs, 3 are disconnected and 3,279 planar.
In \cite{genusILP2} only nonplanar graphs were tested. Note that the set of Rome graphs
not only contains isomorphic graphs,
but even identical copies. Some files also seem obscure: e.g. in \verb+grafo6975.39.graphml+ due to the
otherwise used convention, there should be a graph with 39 vertices. Nevertheless it has 105 vertices. We filtered out the
diconnected graphs, the
planar graphs, and the {\em obscure} ones and -- like \cite{genusILP2} -- received a list of 8,249 nonplanar
graphs for which the genus had to be computed. It should be mentioned that the sizes of the
Rome graphs are a bit misleading when it comes to estimating the complexity of computing the genus:
many of the graphs have vertices of degree 1 and 2, which do not increase the complexity of the computation
of the genus. In \cite{genusILP2} a Xeon Gold 6134 CPU was used to compute the genus of these graphs.
For each graph a time limit of 10 minutes and a memory limit of 8 GB was given. With these restrictions
\ILP was able to decide 82\% of the instances. For multi\_genus the memory consumption is negligible.
On the Core i7 it could determine the genus of 98.57\% of the graphs within the same time limit of 10 minutes and even
with a time limit of only 0.25 seconds for each graph, it can decide  83.7\% of the cases.

In Figure~\ref{fig:random} the development of the running times for random cubic and quartic graphs are given. For all sizes the version for more than 64 vertices using {\em unsigned \_\_int128} was used in order not to have some
misleading behaviour around 64 vertices.
As expected, for given fixed degree of the vertices, the measured times depend exponentially on the number of vertices.

\begin{figure}[h!t]
	\centering
	\includegraphics[width=0.6\textwidth]{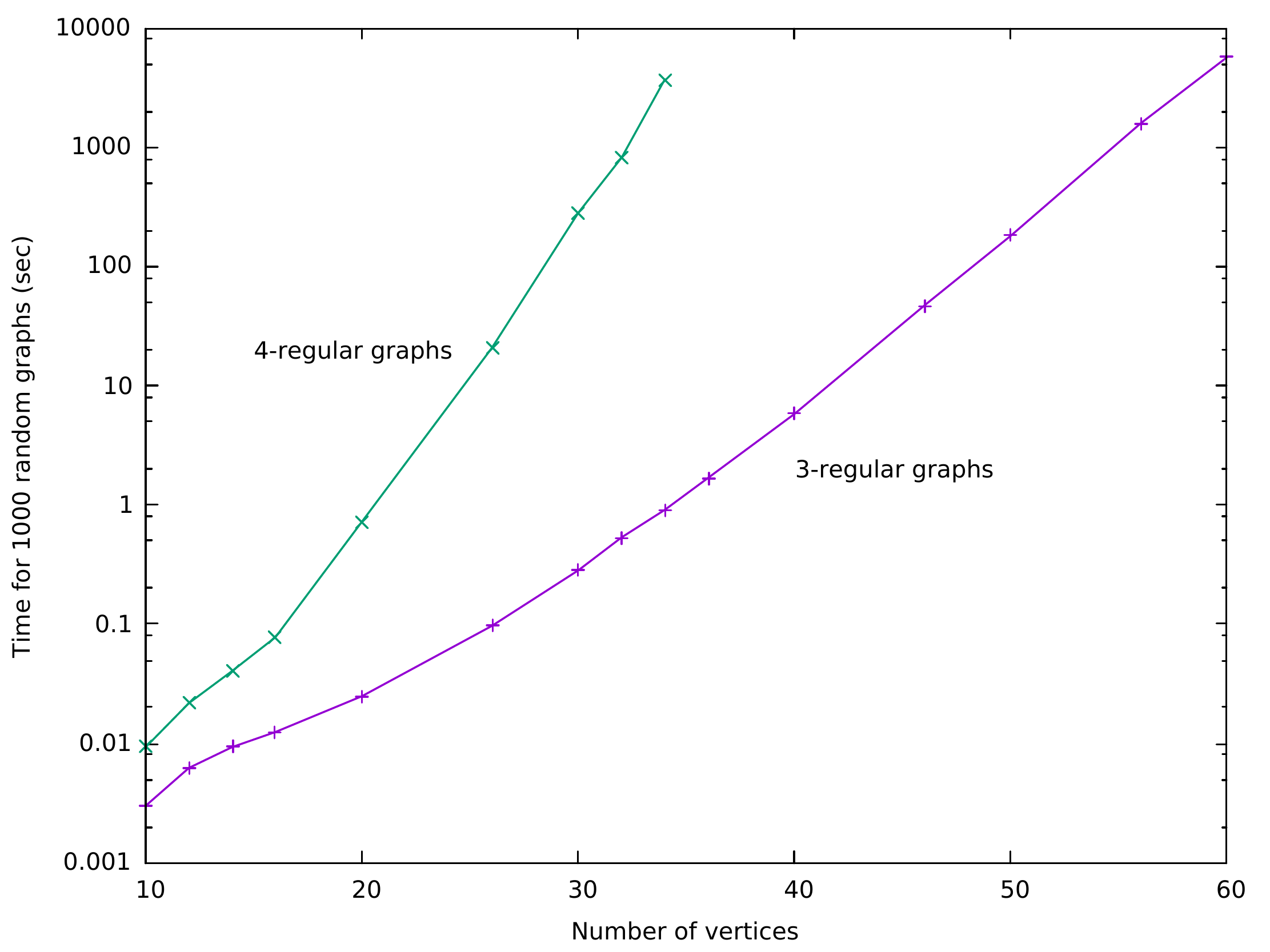}
	\caption{The running times for computing the genus of 1000 random cubic, resp. quartic graphs. Note that the time scale is logarithmic.}
	\label{fig:random}
\end{figure}

If we fix the number of vertices, but vary the number of edges, we get -- as expected -- again an exponential growth,
as shown in Figure~\ref{fig:random_edges}.

\begin{figure}[h!t]
	\centering
	\includegraphics[width=0.6\textwidth]{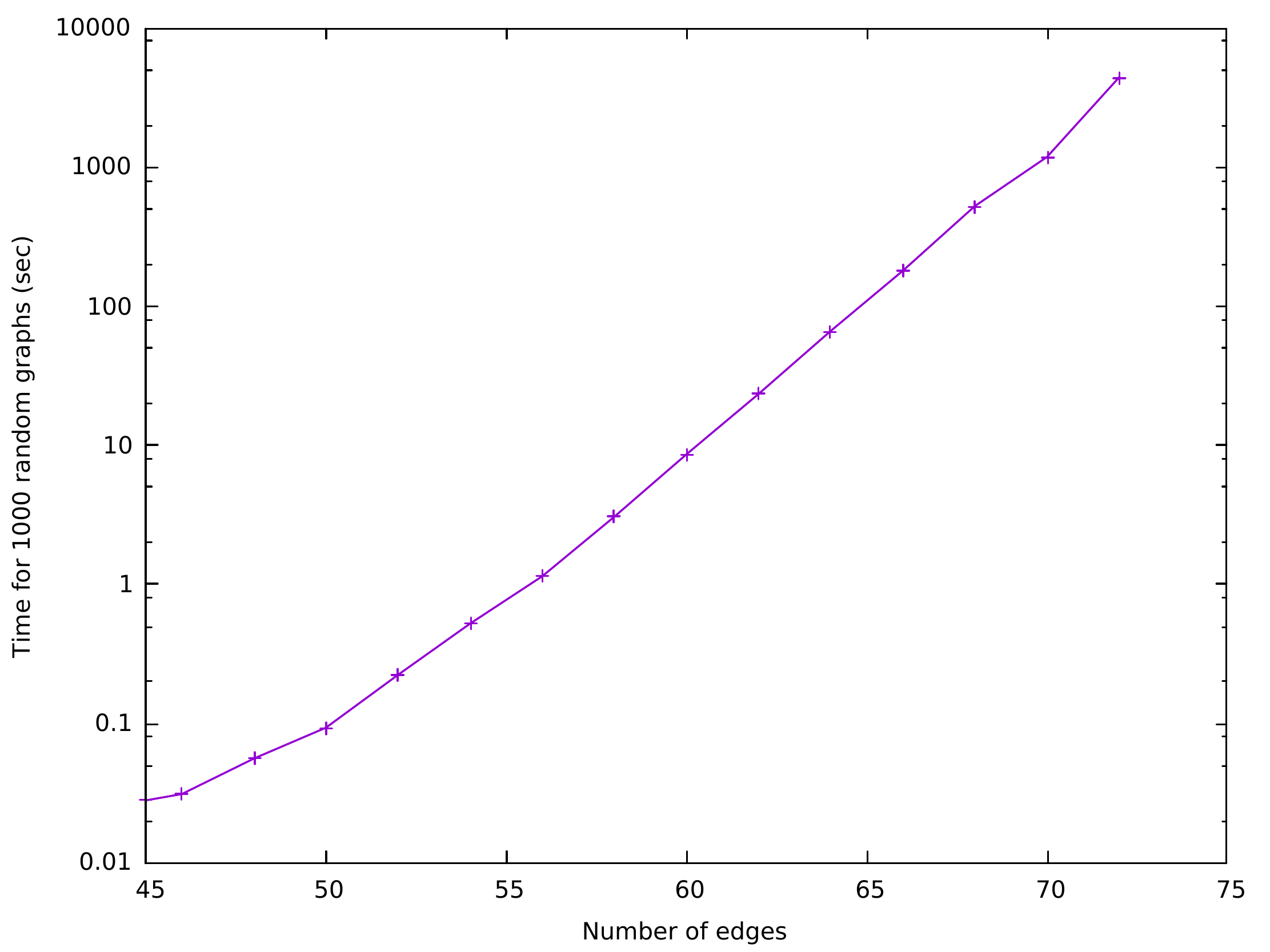}
	\caption{The running times for computing the genus of 1000 random graphs with 32 vertices and a given number of edges. Note that the time scale is logarithmic.}
	\label{fig:random_edges}
\end{figure}

Together with the number of vertices and edges, also the average genus
increases, so it is also interesting to know how the running times develop,
when it is only tested whether the graphs can be embedded in a surface of
given genus -- similar to planarity testing. In Figure~\ref{fig:random_limit} the running times for
testing whether 1000 random cubic resp. quartic graphs can be embedded
in a surface of genus at most 3 are given.  In fact for practically
all of the larger graphs tested, the answer was {\em no}. Nevertheless it is astonishing that from a certain point on
the time necessary to test a graph decreases again.

\begin{figure}[h!t]
	\centering
	\includegraphics[width=0.6\textwidth]{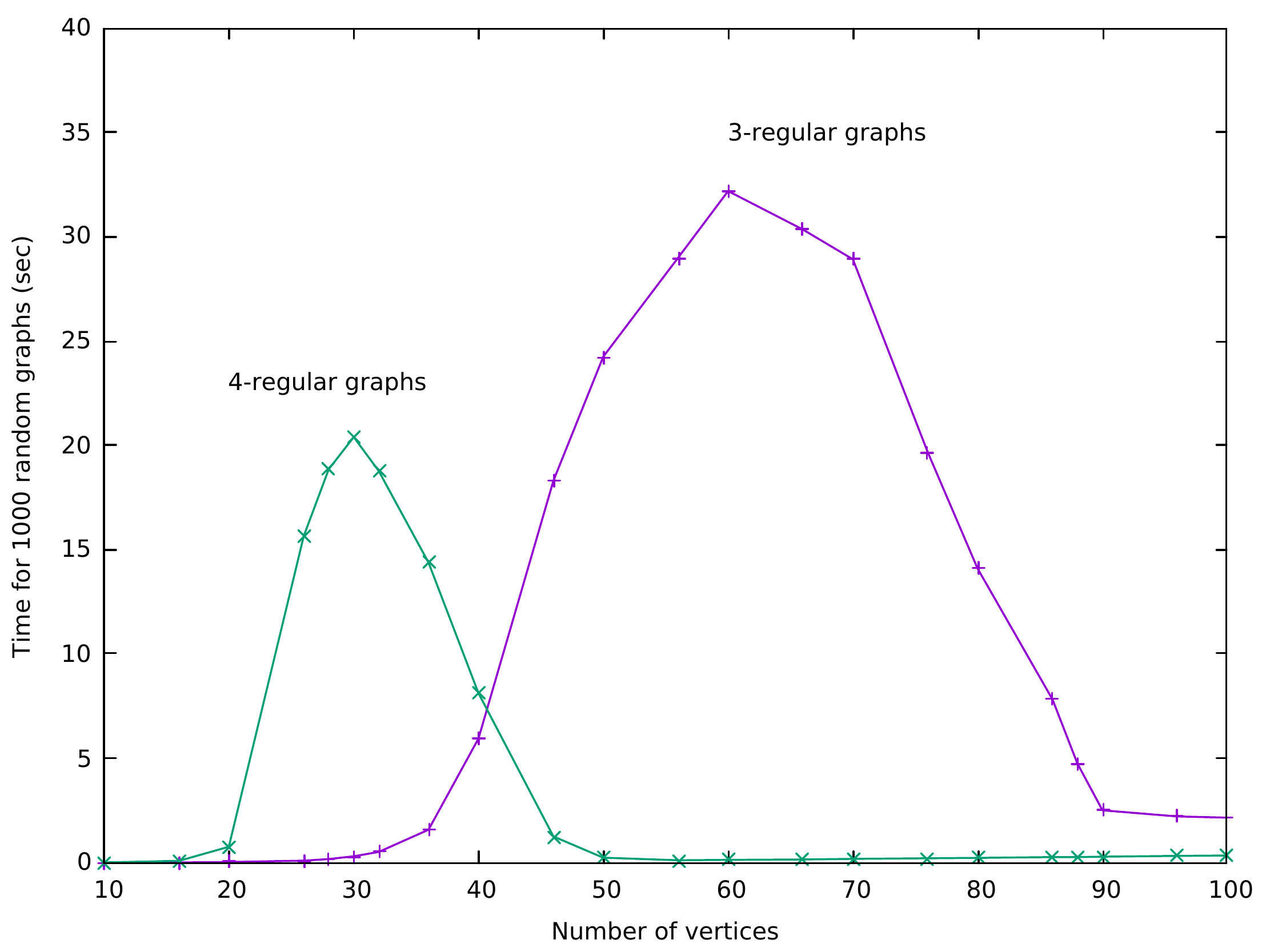}
	\caption{The running times for testing 1000 random cubic, resp. quartic graphs for having genus at most 3.}
	\label{fig:random_limit}
\end{figure}

If one wanted to apply the algorithm to perform very well when testing only planarity, one could apply the reasoning of 
Demoucron, Malgrange and Pertuiset \cite{Demoucron_planar} as soon as a spanning tree and the first cycle is formed.
Of course that would not be useful, as specialized and very efficient algorithms for planarity testing exist.
Although no special adaptation for the planar case is implemented and this is not the task the algorithm was
developed for, a comparison with specialized programs
for planarity testing might be interesting. In \cite{wendy_linear_planar} a practical linear time algorithm for
planarity testing is presented. In the program {\em planarg} \cite{gtools} an implementation 
of this algorithm by Paulette Lieby is available.

Testing 100,000 random cubic graphs on 50 vertices with planarg and
multi\_genus with 0 as an upper bound for the genus, the times are
2.29 (planarg) resp.\ 5.5 (multi\_genus) seconds.  Testing 100,000 random
cubic graphs on 100 vertices, planarg is more than 4 times faster than
multi\_genus (5.0 seconds to 22.1 seconds). All these graphs were
non-planar.  Filtering the 11,529 rome graphs for non-planar ones took
0.36 seconds with planarg and 0.29 seconds with multi\_genus.  In this
case about 28\% of the graphs were planar. In order to test how fast
embeddings are found if they exist, once sparse planar graphs -- that
is cubic graphs -- were tested and once dense graphs -- that is
triangulations. For these tests, for multi\_genus no lower bound for the genus was computed.
We used the 285,914 fullerenes on 100 vertices and
their duals (triangulations on 52 vertices) randomly relabeled in
order to avoid an impact of the embedding produced by the generation
program.  For the fullerenes, planarg needed 19.5 seconds (compared
to 19.1 seconds for multi\_genus) and for the triangulations planarg
needed 11.85 seconds (compared to 11.1 seconds for multi\_genus).

\section*{Testing}

In order to test the implementation, independent programs were used and the results were compared
to the results of multi\_genus. For all cases tested the results agreed.

In order to test the option to compute the genus, the genus of each (connected) graph in the following set was computed by
multi\_genus and  MinGenusEmbedder and the result was compared. The sets of graphs are: all graphs on up to 11 vertices,
all 3-regular graphs on up to 24 vertices,
all 4-regular graphs on up to 16 vertices,
all 5-regular graphs on up to 14 vertices,
all 3-regular graphs on up to 26 vertices with girth at least 5,
all 3-regular graphs on up to 28 vertices with girth at least 6,
all 3-regular graphs on up to 34 vertices with girth at least 7,
all 3-regular graphs on up to 44 vertices with girth 8,
all (3,9)-cages -- that is all 3-regular graphs on 58 vertices with girth 9,
all 4-regular graphs on up to 24 vertices with girth 5, and finally
all graphs with valence vector $(2,2,3,3,3)$.

In order to test the option that makes multi\_genus generate all embeddings of a certain genus,
a simple independent program was implemented that generated all combinations of
all vertex orders around the vertices. Filtering the embeddings generated this way
for those with a given genus we had a very slow but independent test. Then for each graph in one of the following sets,
the range of possible genera was computed by the Euler formula and for each graph 
and each possible genus the embeddings of this genus were independently generated and the 
number of embeddings as well as the number of non-isomorphic embeddings (computed
by an isomorphism checking program using lists) were compared. Due to the enormous
number of embeddings already for small graphs, not too many graphs and no large graphs could be
tested. The sets of graphs are: all 3-regular graphs on up to 18 vertices,
all graphs on 7 vertices with 6 to 17 edges,
all graphs on 8 vertices with 14 and with 15 edges,
all graphs with valence vector $(1,1,1,5)$, and finally
all graphs with valence vector $(0,2,3,2,3)$ and girth at least 4.

\section*{Results obtained or confirmed by multi\_genus}

Times given for multi\_genus in this section
are again on an {\em Intel Core i7-9700 CPU @ 3.00GHz} running on one core at 4.4-4.7 Ghz.

In  \cite{g-unique} Plummer and Zha prove a theorem describing the cases when $K_{c+1}$ is the unique c-connected graph with smallest genus
-- except for the two cases $c=9$ and $c=13$ which are not decided and posed as an open question. This question is answered in
\cite{dualconnectivity} showing that in these cases the complete graphs are not unique, but that in these cases the graphs $M_{c+2}$ on $c+2$ vertices,
obtained by deleting a maximum matching from $K_{c+2}$, have the same genus as $K_{c+1}$.
The embeddings given in that article were computed by multi\_genus. Computing the genus $g(M_{11})=4$ takes $0.005$ seconds and
computing the genus $g(M_{15})=10$ takes $7$ hours and $6$ minutes.

In \cite{genus_circulants} Conder and Grande determine all circulant graphs of genus 1 and 2. A large part of the proof discusses 12 specific circulant graphs
and in order to prove that 11 of these graphs have genus larger than 2, next to several pages of theoretical argumentation also more than 80 CPU hours
were needed. The program described in \cite{genusILP1} confirms these results in 180 hours of CPU time (without additional theoretical arguments), and \ILP
computes the genera {\em in a matter of seconds} (the exact value isn't given).
Multi\_genus confirms the results of the paper in less than $0.03$ seconds. Computing the exact genera (once genus 2, 7 times genus 3, 3 times genus 4, and once genus 5)
and minimum genus embeddings takes $6.4$ seconds.

In \cite{genus_gray} the genus of the Gray graph is theoretically determined by a nontrivial construction.
\ILP confirms this result within $42$ hours.
Multi\_genus confirms this in $28.3$ seconds. In order to determine all $258,696$ (labeled) minimum genus embeddings ($219$ non-isomorphic),
multi\_genus needed a bit less than $10$ minutes. Isomorphism rejection is done by an independent program simply storing canonical
embeddings in lists.

In \cite{conder_genus} the genus (and also non-orientable genus) of several graphs was determined. They describe four specialized approaches
they apply to some special graphs that have in general a large symmetry group.
One of them -- they call it the {\em subgroup orbit method} -- is 
especially suited for as they write {\em graphs on surfaces with a certain degree of symmetry} and works well for graphs that allow an embedding with a
face-transitive automorphism group. So the approach is far from general and the program is also not available for everybody. Our general approach
cannot reproduce their results for the Hoffman-Singleton graph, the Ljubljana graph or the Iofinova-Ivanov graph -- at least not without an
excessive amount of time and/or special adaptations.
Some of the other examples they give can also be solved and sometimes extended
by our general approach without any manual interference -- just by
piping the graph into multi\_genus. In \cite{conder_genus} in general no precise running times are
given. The instances for which the results could be confirmed and sometimes extended are:

\begin{description}

\item{The graph $C_3\square C_3 \square C_3$:} In \cite{conder_genus} it says that {\em with a natural vertex labeling} the subgroup orbit method
  {\em takes only a couple of minutes} to find a genus 7 embedding. The method described here takes -- depending on the labeling
  -- from $0.19$ seconds to $281$ seconds to find an embedding. Of course there may also be labelings that take even less or even more time.
  In total there are 188,211,024 minimum genus embeddings, 145,468 of them pairwise non-isomorphic, but computing these took
  almost 3 weeks of CPU time (on another, much older, machine used for the large memory available for isomorphism rejection).
 
  Not only constructing a genus 7 embedding, but also proving its minimality by excluding the existence of an embedding of smaller genus
  takes between $1.5$ and $4$ hours depending on the labeling (both with BFS numbering first). 

\item{The Tutte graph (or $(3,8)$-cage):} In  \cite{conder_genus} no running times are given, but they construct a genus $4$ embedding with cyclic automorphism group of order $3$. The present approach takes $0.005$ seconds to determine the genus as $4$ and $0.13$ seconds to construct all
  $13,440$ embeddings. Among these embeddings there are $15$ non-isomorphic embeddings -- $4$ with group size $1$, $10$ with group size $2$
  ($2$ of them allowing a reflection) and one with group size $3$.

\item{The Gray graph:} The running times for the Gray graph were already given. In  \cite{conder_genus} it is reported that there are minimum genus
  embeddings with an automorphism group of order 6. Checking all possible embeddings, the result is that there are $186$ non-isomorphic embeddings
  with trivial symmetry, $23$ with group size $2$, $4$ with group size $3$, $4$ with group size $6$, and $2$ with group size $18$.

\item{The Folkman graph:} For the Folkman graph, in \cite{conder_genus} minimum genus embeddings with group size $8$ are constructed.
  The method described here takes less than $0.001$ seconds to determine the genus as $3$ and $0,037$ seconds to construct all $7,680$
  minimum genus embeddings. Among these embeddings there are $7$ pairwise non-isomorphic, $2$ with group size $2$, $3$ with group size $4$,
  and $2$ with group size $8$. All groups contain a reflection.

\item{The Doyle-Holt graph:} For the Doyle-Holt graph  \cite{conder_genus} describes a genus $5$ embedding with an automorohism group of size $2$.
  The present approach needs $0.23$ seconds to determine the genus of the graph and $7.3$ seconds to determine all $1,107$ minimum genus embeddings.
  There are $24$ pairwise non-isomorphic embeddings -- $17$ with trivial group and $7$ with group size 2.

\item{The dual Menger graph of the Gray configuration:} For this 6-regular graph on 27 vertices, the present approach needs $20$ seconds to determine the genus as $6$. In a bit more than $26$ minutes it constructed all $216$ minimum genus embeddings -- which turned out to be isomorphic. So the minimum genus
  embedding of the dual Menger graph of the Gray configuration is unique. It has automorphism group size $6$.

  \end{description}

\bibliographystyle{plain}
\bibliography{../literatur.bib}

\end{document}